\documentclass[11pt]{article}
\usepackage{amsmath}
\usepackage{amsfonts}
\usepackage{mathtools}
\usepackage{enumerate}
\usepackage{amsthm}
\usepackage[textwidth=5.65in,textheight=8in]{geometry}
\newtheorem{theorem}{Theorem}[section]
\newtheorem{corollary}[theorem]{Corollary}
\newtheorem{lemma}[theorem]{Lemma}

\newtheorem{definition}[theorem]{Definition}

\newtheorem*{assumption}{Assumption}
\numberwithin{equation}{section}
\def\a{\alpha}				
\def\e{\epsilon}				
\def\va{v_{\a}}				
\def\tu{\tilde{u}}			
\def\E{\mathbb{E}}			
\def\PR{\mathbb{P}}			
\def\R{\mathbb{R}}			
\def\U{\mathbb{U}}			
\def\Rp{\R^{+}}				
\def\bR{\overline{\R}}		
\def\X{\mathbb{X}}			
\def\A{\mathbb{A}}			
\def\K{\mathbb{K}}			
\def\PS{\Pi}					
\def\argmin{\arg\min}		
\def\Z{\mathbb{Z}}			
\def\N{\mathbb{N}}			
\title{
On the Optimality Equation for Average Cost Markov Decision Processes
and its Validity for Inventory Control}
\author{
Eugene~A.~Feinberg, Yan~Liang \\
\small
\textit{Department of Applied Mathematics and Statistics} \\
\small
\textit{Stony Brook University, Stony Brook, NY 11794} \\
\small
\textit{eugene.feinberg@stonybrook.edu, yan.liang@stonybrook.edu}
}
\date{}
\begin{document}
\maketitle
%
%
\begin{abstract}
\textit{
	As is well known, average-cost optimality inequalities imply the existence
	of stationary optimal policies for Markov Decision Processes with average costs per unit time, and these
	inequalities hold under broad natural conditions. This paper provides
	sufficient conditions for the validity of the average-cost optimality equation
	for an infinite state problem with weakly continuous transition probabilities and with possibly unbounded one-step costs and noncompact action sets.
    These conditions also imply the convergence of sequences of discounted relative value functions to  average-cost relative value functions and the continuity of average-cost relative value functions.
    As shown in the paper, the classic periodic-review inventory control problem satisfies these conditions.
	Therefore, the optimality inequality holds in the form of an equality with a continuous average-cost relative value function for this problem. In addition, the $K$-convexity of discounted relative value functions and their convergence to average-cost relative value functions, when the discount factor increases to 1,  imply the $K$-convexity of average-cost relative value functions. This implies that average-cost optimal $(s,S)$ policies for the inventory control problem can be derived from the average-cost optimality equation.
}
\end{abstract}
%
%
\section{Introduction}
\label{sec:introduction}

For Markov Decision Processes (MDPs) with average costs per unit time, the existence of stationary optimal policies follows from the validity of the average-cost optimality inequality (ACOI).  Feinberg et al.~\cite{FEINBERG2012} established broad sufficient conditions for the validity of ACOIs for MDPs with weakly continuous transition probabilities and possibly noncompact action sets and unbounded one-step costs.  In particular, these and even stronger conditions hold for the classic periodic-review inventory control problem with backorders; see Feinberg~\cite{Ftut} or Feinberg and Lewis~\cite{FEINBERG2015}.
Previously, Sch\"{a}l~\cite{Sch93} established sufficient conditions for the validity of ACOIs for MDPs with compact action sets and possibly unbounded costs.  Cavazos-Cadena~\cite{Cadena1991} provided an example in which the ACOI holds but the average-cost optimality equation (ACOE) does not.
This paper presents sufficient conditions for the validity of ACOEs for MDPs with infinite state spaces, weakly continuous transition probabilities and possibly noncompact action sets and unbounded one-step costs  and, by showing that the classic periodic-review inventory control problems satisfy these conditions, establishes the validity of the ACOEs for the inventory control problems.

Sufficient conditions for the validity of ACOEs for discrete-time MDPs with countable and general state spaces with setwise continuous transition probabilities
are described in  Sennott~\cite[Section 7.4]{SENNOTT1998}, \cite{SENNOTT2002} and
Hern\'{a}ndez-Lerma and Lasserre~\cite[Section 5.5]{LERMA1996}, respectively.  Zheng~\cite{Zheng} used Sennott's results to provide a simple proof of the optimality of $(s,S)$ policies for average-cost periodic-review inventory control problems with discrete demand.

Ja\'{s}kiewicz and Nowak~\cite{jn06} considered MDPs with Borel state space, compact action sets, weakly continuous transition probabilities and unbounded costs. The geometric ergodicity of
transition probabilities is assumed in Ja\'{s}kiewicz and Nowak~\cite{jn06} to ensure the
validity of the ACOEs. Costa and Dufour~\cite{cd12}  studied the validity of ACOEs for MDPs with Borel state and action spaces,
weakly continuous transition probabilities, which are positive Harris
recurrent, and with possibly noncompact action sets and unbounded costs. Neither the geometric ergodicity nor positive Harris recurrent conditions hold for the periodic-review inventory control problem.

An early attempt, to establish for problems with weakly continuous transition probabilities the results on the validity of the ACOE similar to the results in Hern\'{a}ndez-Lerma and Lasserre~\cite[Section 5.5]{LERMA1996} for problems with setwise continuous transition probabilities,  was undertaken in Montez-de-Oca~\cite{MDA}.  However, the formulations and proofs there, as well as some proofs in \cite{cd12}, relied on a technically incorrect paper with statements contradicting  a counterexample in Luque-Vasques and  Hernández-Lerma~\cite{LHL} relevant to  Berge's maximum theorem.

Section \ref{sec:model definition} of this paper  describes the general MDPs framework.
In particular, it states Assumptions (\textbf{W*}) and (\textbf{B}) from Feinberg et al.~\cite{FEINBERG2012}, which guarantee the validity of the ACOIs.
Section \ref{sec:main result} provides the sufficient conditions for the validity of the ACOEs, which extends the sufficient conditions in Hern\'{a}ndez-Lerma and Lasserre~\cite[Theorem 5.5.4]{LERMA1996} to weakly continuous transition probabilities.
By verifying these conditions,  it is shown in Section~\ref{sec:inventory control}, that the ACOE holds for the classic periodic-review inventory control problems with general demands.
The paper also establishes $K$-convexity and continuity of the average-cost relative value  function and shows that an optimal $(s,S)$ policy can be derived from the ACOE. It also shows that at the level $s$ there are at least two optimal decisions: do not order and order up to the level $S.$

\section{Model definition}
\label{sec:model definition}

Consider a discrete-time MDP with a state space $\X,$ an action space $\A,$ one-step
costs $c,$ and transition probabilities $q.$ Assume that $\X$ and $\A$ are Borel subsets
of Polish (complete separable metric) spaces.

Let  $c(x,a):\X\times\A\to\bR=\R\cup\{+\infty\}$ be the one-step cost and $q(B|x,a)$ be the transition kernel
representing the probability that the next state is in $B\in\mathcal{B}(\X),$
given that the action $a$ is chosen in the state $x.$

We recalled that a function $f:U\to \R\cup \{+\infty\}$ for a metric space $\U,$ where $U$ is a subset of a metric space $\U,$ is
called inf-compact, if for every $\lambda\in\R$ the level set $\{u\in\U :f(u)\leq \lambda \}$
is compact.
\begin{definition}\rm{(Feinberg et al.~\cite[Definition 1.1]{FEINBERG2013}, Feinberg~\cite[Definition 2.1]{Ftut})}
\label{def:k inf compact}
	A function $f:\X\times\A\to \bR$ is called $\K$-inf-compact, if for
	every nonempty compact subset $K$ of $\X$ the function $f:K\times\A\to \bR$ is inf-compact.
\end{definition}
Let the one-step cost function $c$ and transition probability $q$ satisfy the following condition.
\begin{assumption}[\textbf{W*}]

(i) c is $\K$-inf-compact and bounded below, and

(ii) the transition probability $q(\cdot|x,a)$ is weakly continuous in
  $(x,a)\in \X\times\A,$  that is, for every bounded continuous function $f:\X\to\R,$ the function
  $\tilde{f}(x,a):=\int_\X f(y)q(dy|x,a)$ is continuous on $\X\times\A.$

\end{assumption}

The decision process proceeds as follows: at each time epoch
$t=0,1,\dots,$ the current state of the system, x, is observed.
A decision-maker chooses an action a, the cost c(x,a) is
accrued, and the system moves to the next state according to
$q(\cdot|x,a).$ Let $H_t = (\X\times\A)^{t}\times\X$
be the set of histories for $t=0,1,\dots\ .$ Let $\PS$ be the set of all policies. A
(randomized) decision rule at period $t=0,1,\dots$ is a regular transition probability
$\pi_t : H_t\to \A,$ 
that is, (i) $\pi_t(\cdot|h_t)$ is a probability distribution on $\A,$
where $h_t=(x_0,a_0,x_1,\dots,a_{t-1},x_t),$ and (ii) for any measurable subset
$B\subset \A,$ the function $\pi_t(B|\cdot)$ is measurable on $H_t.$
A policy $\pi$ is a sequence $(\pi_0,\pi_1,\dots)$ of decision rules.
Moreover, $\pi$ is called non-randomized if each probability measure $\pi_t(\cdot|h_t)$ is
concentrated at one point. A non-randomized policy is called stationary if all decisions depend
only on the current state.

The Ionescu Tulcea theorem implies that an initial state $x$ and a policy $\pi$ define a unique
probability $\PR_{x}^{\pi}$ on the set of all trajectories $\mathbb{H}_{\infty}=(\X\times\A)^{\infty}$ endowed with the product of $\sigma$-field defined by Borel $\sigma$-field of $\X$ and $\A;$ see Bertsekas and Shreve~\cite[pp. 140--141]{BERTSEKAS1996} or Hern\'{a}ndez-Lerma and Lasserre~\cite[p. 178]{LERMA1996}.
Let $\mathbb{E}_{x}^{\pi}$ be an expectation with respect to $\PR_{x}^{\pi}.$

For a finite-horizon $N=0,1,\dots,$ let us define the expected total discounted costs,
\begin{equation}\label{eqn:sec_model def:finite total disc cost}
    v_{N,a}^{\pi} := \mathbb{E}_{x}^{\pi} \sum_{t=0}^{N-1}
    \alpha^{t} c(x_t,a_t), \;\; x\in\X,
\end{equation}
where $\alpha\in [0,1)$ is the discount factor and $v_{0,a}^{\pi} (x) = 0.$
When $N=\infty,$ 
equation (\ref{eqn:sec_model def:finite total disc cost}) defines an
infinite-horizon expected total discounted cost denoted by $v_{\a}^{\pi}(x).$ Let $v_\alpha:=\inf_{\pi\in\PS} v_\a^\pi(x),$ $x\in\X.$  A policy $\pi$ is called optimal for the discount factor $\a$ if $v^\pi_\a(x)=v_\a(x)$ for all $x\in\X.$

The \emph{average cost per unit time} is defined as
\begin{equation}\label{eqn:sec_model def:avg cost}
    w^{\pi}(x):=\limsup_{N\to +\infty} \frac{1}{N}v_{N,1}^{\pi} (x), \;\; x\in\X.
\end{equation}

Define the optimal value function $w(x):=\inf_{\pi\in\Pi} w^{\pi} (x),$  $x\in\X.$
 A policy $\pi$ is called
average-cost optimal if $w^{\pi}(x)=w(x)$ for all $x\in\X.$

Let
\begin{align}\label{defmauaw}
	\begin{split}
  		m_{\alpha}: = \underset{x\in\X}{\inf} v_{\alpha}(x), & \quad
  		u_{\alpha}(x): = v_{\alpha}(x) - m_{\alpha}, \\
  		\underline{w}: = \underset{\alpha\uparrow1}{\liminf}(1-\alpha)m_{\alpha}, & \quad
  		\bar{w}: = \underset{\alpha\uparrow1}{\limsup}(1-\alpha)m_{\alpha}
	\end{split}
\end{align}
The function $u_\alpha$ is called the discounted relative value function. Assume that the following assumption holds in addition to
Assumption (\textbf{W*}).

\begin{assumption}[\textbf{B}]

(i) $w^{*} := \inf_{x\in\X} w(x)< +\infty,$ and
(ii) $\underset{\alpha\in [0,1)}{\sup} u_{\alpha}(x) < \infty,$ $x\in\X.$

\end{assumption}

As follows from Sch\"al~\cite[Lemma 1.2(a)]{Sch93}, Assumption~\textbf{B}(i) implies that $m_\alpha<+\infty$ for all $\alpha\in [0,1).$  Thus, all the quantities in \eqref{defmauaw} are defined.  According to Feinberg et al.~\cite[Theorem 4]{FEINBERG2012}, if Assumptions
(\textbf{W*}) and (\textbf{B}) hold, then $\underline{w} = \bar{w}.$ In addition, for each
sequence $\{ \a_n \}_{n=1,2,\ldots}$ such that $\a_n\uparrow1$ as $n\to +\infty,$
\begin{align}
	\lim_{n\to +\infty} (1-\a_n) m_{\a_n} = \underline{w} = \bar{w}.
\label{eqn:bar w = underline w}
\end{align}

Define the following function on $\X$ for the sequence $\{ \a_n\uparrow 1 \}_{n=1,2,\ldots} :$
\begin{eqnarray}\label{EQN1}
  \tu(x) &:=& \liminf_{n\to +\infty,y\rightarrow x} u_{\a_n}(y) .
\end{eqnarray}
In words, ${\tilde u}(x)$ is the largest number such that ${\tilde u}(x)\le \liminf_{n\to\infty}u_{\alpha_n}(y_n)$ for all sequences $\{y_n\to x\}.$ Since $u_{\a}(x)$ is nonnegative by definition, then $\tu(x)$ is also nonnegative. The function $\tilde u,$ defined in \eqref{EQN1} for a sequence $\{ \a_n\uparrow 1 \}_{n=1,2,\ldots}$ of nonnegative discount factors, is called an average-cost relative value function.

\section{Average cost optimality equation}
\label{sec:main result}

If Assumptions (\textbf{W*}) and (\textbf{B}) hold, then,
according to Feinberg et al.~\cite[Corollary 2]{FEINBERG2012}, there exists a
stationary policy $\phi$ satisfying
\begin{align}
	\underline{w} + \tu(x) &\geq c(x,\phi (x)) + \int_{\X} \tu(y)q(dy|x,\phi (x)) ,
    \qquad x\in\X,
\label{eqn:ACOI}
\end{align}
with $\tu$ defined in \eqref{EQN1} for an arbitrary sequence $\{ \a_n\uparrow 1 \}_{n=1,2,\ldots},$ and
\begin{align}\label{eqsfkz}
    w^{\phi}(x)= \underline{w} = \lim_{\alpha\uparrow 1}
    (1-\alpha)v_{\alpha} (x) = \bar{w} = w^*, \qquad x\in\X.
\end{align}
These equalities imply that the stationary policy $\phi$ is average-cost optimal and $w^\phi(x)$ does not depend on $x.$

Inequality (\ref{eqn:ACOI}) is known as the ACOI.
We remark that a weaker form of the ACOI with $\underline{w}$ substituted with $\bar{w}$ is also described in Feinberg et al.~\cite{FEINBERG2012}.  If Assumptions (\textbf{W*}) and (\textbf{B}) hold, let us define $w:=\underline{w};$ see \eqref{eqsfkz} for other equalities for $w.$

Recall the following definition.
\begin{definition}[{Hern\'{a}ndez-Lerma and Lasserre~\cite[Remark 5.5.2]{LERMA1996}}]
	A family $\mathcal{H}$ of real-valued functions on a metric space $X$
	is called equicontinuous at the point $x\in X$ if for each $\e > 0$
	there exists an open set $G$ containing $x$ such that
	\begin{equation*}
		| h(y) - h(x) | < \e \;\;\; \text {for all}\  y\in G \text{ and for all } h\in\mathcal{H}.
	\end{equation*}
	The family $\mathcal{H}$ is called equicontinuous (on $X$) if it is equicontinuous at
	all $x \in X.$
	\label{def:equicontinuous}
\end{definition}

The following theorem provides sufficient conditions under which there exist a stationary policy $\phi$ and a function
$\tu(\cdot)$ satisfying the ACOE.  This theorem is similar to Theorem 5.5.4 in Hern\'{a}ndez-Lerma and Lasserre~\cite{LERMA1996}, where MDPs with setwise continuous transition probabilities are considered.

\begin{theorem}
Let Assumptions (\textbf{W*}) and (\textbf{B}) hold. Consider a sequence $\{ \a_n\uparrow1 \}_{n=1,2,\ldots}$ of nonnegative discount factors.
If the sequence $\{u_{\a_n}\}_{n=1,2,\ldots}$ is equicontinuous and there exists a nonnegative measurable function $U(x),$ $x\in\X,$ such that
$U(x)\geq u_{\a_n}(x),$ $n=1,2,\ldots,$ and $\int_{\X} U(y)q(dy|x,a) < +\infty$ for all $x\in\X$ and $a\in \A,$
then the following statements hold.
\begin{enumerate}[(i)]
	\item There exists a subsequence $\{ \a_{n_k} \}_{k=1,2,\ldots}$ of $\{ \a_n \}_{n=1,2,\ldots}$ such that $\{u_{\a_{n_k}}(x)\}$ converges pointwise to $\tu(x),$ $x\in\X,$ where $\tu (x)$ is defined in \eqref{EQN1} for the sequence $\{ \a_{n_k}\}_{k=1,2,\ldots}.$ In addition, the function $\tu(x)$ is continuous. \label{stat:1}
	\item There exists a stationary  policy $\phi$ satisfying the ACOE with the nonnegative function $\tilde u$ defined for the subsequence  $\{ \a_{n_k} \}_{k=1,2,\ldots}$ mentioned in statement \eqref{stat:1}, that is, for all $x\in\X,$
     \begin{align}
  		w + \tu(x) = c(x,\phi(x)) + \int_{\X} \tu(y)q(dy|x,\phi(x)) =
  		\min_{a\in \A} [ c(x,a) + \int_{\X} \tu(y)q(dy|x,a) ],
  	\label{EQN11}
	\end{align}
and, since the left equation in \eqref{EQN11} implies inequality \eqref{eqn:ACOI}, every stationary policy satisfying \eqref{EQN11} is average-cost optimal.
\end{enumerate}
\label{thm:acoe}
\end{theorem}

\begin{proof} (i)
Since $\{u_{\a_n}\}_{n=1,2,\ldots}$ is equicontinuous and Assumption (\textbf{B}) holds, then, according to the
Arzel\`{a}-Ascoli theorem, there exist a subsequence $\{ \a_{n_k} \}_{k=1,2,\ldots}$ of $\{\a_n\}_{n=1,2,\ldots}$ and a continuous function $\tu^*(\cdot)$ such that
\begin{align}
	\lim_{k\to +\infty} u_{\a_{n_k}}(x) = \tu^*(x) , \quad\quad x\in\X,
	\label{eqn:convergence:1}
\end{align}
and the convergence is uniform on each compact subset of $\X.$

Consider the function $\tu(x)$ defined in (\ref{EQN1}) for the sequence $\{\a_{n_k} \}_{k=1,2,\ldots}.$
In view of the definition of the function $\tu(x),$ there exist a subsequence $\{ \tilde{\a}_l \}_{l=1,2,\ldots}$ of
$\{ \a_{n_k} \}_{k=1,2,\ldots}$ and a sequence $\{y_l\}_{l=1,2,\ldots}\subset\X$ such that $\tilde{\a}_l\uparrow 1,$
$y_l\to x$ as $l\to +\infty$ and
\begin{align}
	\lim_{l\to\infty} u_{\tilde{\a}_l} (y_{l}) = \tu (x) .
	\label{eqn:limit-1}
\end{align}
In addition, the family $\{u_{\tilde{\a}_l}(x)\}_{l=1,2,\ldots}$ is also equicontinuous for all $x\in\X.$

For any fixed $\epsilon>0,$ (\ref{eqn:convergence:1}) implies that there exists a constant
$N_1>0$ such that for all $l\geq N_1$
\begin{align}
	|\tu^*(x) - u_{\tilde{\a}_l}(x)| < \e /3 .
	\label{eqn:ub part 1}
\end{align}
Since the family $\{u_{\tilde{\a}_l}\}_{l=1,2,\ldots}$ is equicontinuous, then there exist a constant $N_2>0$ and a neighborhood $B(x)$ of $x$ such that, for all $l\geq N_2$ and $y_{l}\in B(x),$
\begin{align}
	|u_{\tilde{\a}_l}(x) - u_{\tilde{\a}_l}(y_l)| < \e /3 .
	\label{eqn:ub part 2}
\end{align}
In view of \eqref{eqn:limit-1}, there exists $N_3>0$ such that for all $l\geq N_3,$
\begin{align}
	|u_{\tilde{\a}_l}(y_l) - \tu(x)| < \e /3 .
	\label{eqn:ub part 3}
\end{align}
Then \eqref{eqn:ub part 1}, \eqref{eqn:ub part 2}, and \eqref{eqn:ub part 3} imply that
for all $l\geq \max \{ N_1,N_2,N_3 \},$
\begin{align}
	\begin{split}
		& |\tu^*(x) - \tu(x)|
		\leq  |\tu^*(x) - u_{\tilde{\a}_l}(x)| +
		| u_{\tilde{\a}_l}(x) - u_{\tilde{\a}_l}(y_l)| +
		| u_{\tilde{\a}_l}(y_l) - \tu(x)| \\
		&< \e /3+\e /3+\e /3=\e .	
	\end{split}
	\label{eqn:ub}
\end{align}
Since $\e>0$ can be chosen arbitrarily, then \eqref{eqn:convergence:1} and \eqref{eqn:ub}
imply that
\begin{align}\label{eqn:convergence:2}
	\lim_{k\to +\infty} u_{\a_{n_k}}(x) = \tu^*(x) = \tu(x) , \quad\quad x\in\X.
\end{align}

(ii) Since Assumptions (\textbf{W*}) and (\textbf{B}) hold, then according to Feinberg et al.~\cite[Corollary 2]{FEINBERG2012}, there exists a stationary policy $\phi$ satisfying the ACOI with $\tu$ defined in \eqref{EQN1} for the sequence $\{ \a_{n_k} \}_{k=1,2,\ldots},$ that is
\begin{align}
	w + \tu(x) \geq c(x,\phi(x)) + \int_{\X} \tu(y)q(dy|x,\phi(x)) .
	\label{eqn:geq}
\end{align}

To prove the ACOE, it remains to prove the opposite inequality to  \eqref{eqn:geq}.
According to Feinberg et al.~\cite[Theorem 2(iv)]{FEINBERG2012}, the discounted-cost optimality equation is $v_{\a_{n_k}}(x) = \min_{a\in \A}
[ c(x,a) + \a \int_{\X} v_{\a_{n_k}} (y) q(y|x,a) ],$ $x\in\X,$
%
which, by subtracting $m_{\a}$ from both sides, implies that for all $a\in\A$
\begin{align}
	(1-\a_{n_k})m_{\a_{n_k}} + u_{\a_{n_k}} (x) \leq 
 c(x,a) + \a \int_{\X} u_{\a_{n_k}} (y) q(y|x,a),\qquad x\in\X  .
	\label{eqn:transform dcoe}
\end{align}
Let $k\to\infty.$ In view of \eqref{eqn:bar w = underline w}, \eqref{eqn:convergence:2}, and Lebesgue's dominated convergence theorem,   \eqref{eqn:transform dcoe} implies that for all $a\in\A$
\[
	w + \tu(x) \leq  c(x,a) + \int_{\X} \tu(y)q(dy|x,a), \qquad x\in\X,
\]
which implies
\begin{align}
	w + \tu(x) \leq \min_{a\in \A} [ c(x,a) + \int_{\X} \tu(y)q(dy|x,a)],\qquad x\in\X .
	\label{eqn:leq}
\end{align}
Since $\min_{a\in \A} [ c(x,a) + \int_{\X} \tu(y)q(dy|x,a)]\leq c(x,\phi(x)) + \int_{\X} \tu (y) q(y|x,\phi(x)),$ then \eqref{eqn:geq} and \eqref{eqn:leq} imply
\eqref{EQN11}.
\end{proof}

\section{Inventory control problem}
\label{sec:inventory control}

Let $\R$ denote the real line, $\Z$ denote the set of all integers, $\Rp:=[0,+\infty)$ and
$\N_0 = \{0,1,2,\ldots \}.$ Consider the classic stochastic
periodic-review inventory control problem with fixed ordering cost and general demand.
At times $t=0,1,\ldots,$ a decision-maker views the current inventory of a single commodity
and makes an ordering decision. Assuming zero lead times, the products are immediately
available to meet demand. Demand is then realized, the decision-maker views the remaining
inventory, and the process continues. The unmet demand is backlogged and the cost of
inventory held or backlogged (negative inventory) is modeled as a convex function. The
demand and the order quantity are assumed to be non-negative. The state and action spaces
are either (i) $\X = \R$ and $\A = \Rp,$ or (ii) $\X=\Z$ and $\A = \N_0.$
The inventory control problem is defined by the following parameters.
\begin{enumerate}
    \item $K\geq 0$ is a fixed ordering cost;
    \item $\bar{c}>0$ is the per unit ordering cost;
    \item $h(\cdot)$ is the holding/backordering cost per period, which is assumed to be
    a convex function on $\X$ with real values
    and $h(x)\to\infty$ as $|x|\to\infty;$
    \item $\{D_t,t=1,2,\dots\}$ is a sequence of i.i.d. nonnegative finite
    random variables representing the demand at periods $0,1,\dots\ .$
    We assume that $\mathbb{E} [h(x-D)] < \infty$ for all $x\in\X$ and $\PR(D>0)>0,$
    where $D$ is a random variable with the same distribution as $D_1;$ \label{enum:demand}
    \item $\a \in [0,1)$ is the discount factor.
\end{enumerate}
Note that $\E [D] < \infty$ since, in view of Jensen's inequality,
 $h(x-\E [D])\le\E [h(x-D)]<\infty.$ Without loss of generality,  assume that $h$ is nonnegative and $h(0) = 0.$ The assumption $\PR(D>0)>0$ avoids the
trivial case when there is no demand. If $\PR(D=0)=1,$ then the optimality inequality does not hold because $w(x)$ depends on $x;$ see Feinberg and Lewis~\cite{FEINBERG2015} for details.

The dynamic of the system is defined by the equation
\begin{displaymath}
	x_{t+1} = x_t + a_t - D_{t+1}, \quad t=0,1,2,\dots ,
\end{displaymath}
where $x_t$ and $a_t$ denote the current inventory level and the ordered amount at period $t,$
respectively. Then the one-step cost is
\begin{align}
	c(x,a) = KI_{\{ a > 0 \}} + {\bar c}a + \E[h(x+a-D)],
	\quad (x,a)\in \X\times\A ,
	\label{inventory-control:cost function}
\end{align}
where $I_{\{a>0\}} $ is an indicator of the event $\{a>0\},$

According to Feinberg and Lewis~\cite[Corollary 6.1, Proposition 6.3]{FEINBERG2015}, Assumptions (\textbf{W*}) and (\textbf{B}) hold for the MDP corresponding to the described inventory control problem. This implies that the optimality equation for the total discounted  costs can be written as
\begin{align}
	\va (x) & = \min \{ \min_{a\geq 0}[ K + G_{\alpha}(x+a)], G_{\alpha}(x) \} - \bar{c}x, ,\quad x\in\X , \label{eqn:va}
\end{align}
where
\begin{equation}
G_{\alpha}(x)  := {\bar c} x + \E[h(x-D)] + \alpha \E [v_{\alpha}(x-D)],\label{GNG222}\qquad x\in\X.
\end{equation}
According to Feinberg and Liang~\cite[Theorem 5.3]{FEINBERG2016}, the value function $\va(x)$ is continuous for all $\a\in[0,1).$
The function $G_{\alpha}(x)$ is real-valued (Feinberg and Lewis~\cite[Corollary 6.4]{FEINBERG2015}) and continuous (Feinberg and Liang~\cite[Theorem 5.3]{FEINBERG2016}).

The function $c:\X\times\A\to \R$ is inf-compact; see Feinberg and Lewis~\cite[Corollary 6.1]{FEINBERG2015}. This property and the validity of Assumption (\textbf{W*}) imply that for each $\a\in [0,1)$ the function $v_\alpha$ is inf-compact (Feinberg and Lewis~\cite[Proposition 3.1(iv)]{FEINBERG2007}) and therefore the set $X_\alpha:=\{x\in\X|\, v_\alpha(x)=m_\alpha\},$ where $m_\a$ is defined in \eqref{defmauaw}, is nonempty and compact. The validity of Assumptions (\textbf{W*}) and (\textbf{B}(i)) and the inf-compactness of $c$ imply that there is a compact subset $\cal K$ of $\X$ such that $\X_\alpha\subseteq {\cal K}$ for all $\a\in [0,1);$ Feinberg et. al.~\cite[Theorem 6]{FEINBERG2012}.  Following  Feinberg and Lewis~\cite{FEINBERG2015}, let us consider a bounded interval $[x^*_L,x^*_U]\subseteq \X$ such that
\begin{equation}\label{eq:boundmaxlul}
X_\a\subseteq [x^*_L,x^*_U]\qquad {\rm for\ all\ } \alpha\in [0,1).
\end{equation}

Recall the definitions of $K$-convex functions and $(s,S)$ policies.
\begin{definition}\label{def:k-convex}
	A function $f:\X\to\R$ is called $K$-convex where $K\geq 0,$ if for each $x\leq y$
	and for each $\lambda\in(0,1),$
	\begin{displaymath}
		f((1 - \lambda)x+\lambda y) \leq (1 - \lambda)f(x)+\lambda f(y)+\lambda K	.
	\end{displaymath}
\end{definition}
Suppose $f(x)$ is a continuous $K$-convex function such that
$f(x)\to \infty$ as $|x|\to\infty.$ Let
\begin{align}
	S &\in \underset{x\in\X}{\argmin} \{ f (x) \} \label{eqn:def S}, \\
	s &= \inf \{ x\leq S: f (x) \leq K + f(S) \}. \label{eqn:def s}
\end{align}

\begin{definition}\label{def:sS policy}
	Let $s_t$ and $S_t$ be real numbers such that $s_t\leq S_t,$ $t=0,1,\ldots\ .$
	A policy is called
	an $(s_t,S_t)$ policy at step $t$ if it orders up to the level $S_t,$ if $x_t<s_t,$
	and does not order, if $x_t\geq s_t.$ A Markov policy is called an $(s_t,S_t)$ policy
	if it is an $(s_t,S_t)$ policy at all steps $t=0,1,\ldots\ .$ A policy is called an
	$(s,S)$ policy if it is stationary and it is an $(s,S)$ policy at all steps
	$t=0,1,\ldots\ .$
\end{definition}

Define
\begin{equation}\label{alphastar}
\alpha^*:=1+\lim_{x\to-\infty} \frac{h(x)}{\bar{c} x},
\end{equation}
where the limit exists and  $\alpha^*<1$ since the function $h$ is convex,  see Feinberg and Liang~\cite{FEINBERG2016}.  

\begin{theorem}[{Feinberg and Liang~\cite[Theorem 4.4(i) and Corollary 5.4]{FEINBERG2016}}]\label{thm:fli}
If $\alpha\in (\alpha^*,1)$ is a nonnegative discount factor, then an $(s_\alpha,S_\alpha)$ policy is optimal for the discount factor $\alpha,$ where the real numbers $S_\alpha$ and $s_\alpha$ satisfy \eqref{eqn:def S} and are defined in   \eqref{eqn:def s} respectively with $f(x)=G_\alpha(x),$ $x\in\X.$ The stationary policy $\varphi$ coinciding with the $(s_\alpha,S_\alpha)$ policy at all $x\in\X,$ except $x=s_\a,$ where   $\varphi(s_\a)=S_\alpha-s_\alpha,$ is also optimal for the discount factor $\a.$

\end{theorem}

As shown in Feinberg and Lewis~\cite[Equations (6.20), (6.23)]{FEINBERG2015}, the optimality inequality can be written as
\begin{align}
	w + \tu (x) \geq \min \{ \min_{a\geq 0}[ K + H(x+a)], H(x) \} - \bar{c}x,
	\label{eqn:IC ACOE-1}
\end{align}
where
\begin{align}
	H(x) := \bar{c}x +\E[h(x-D)] + \E[\tu (x-D)].
	\label{eqn:IC ACOE-2}
\end{align}
The following statement is  Theorem 6.10(iii) from Feinberg and Lewis~\cite{FEINBERG2015} with the value of $\alpha^*$ is provided in \eqref{alphastar}; see Theorem~\ref{thm:fli}.
\begin{theorem}
\label{thm:fle}
For each nonnegative $\alpha\in
          (\alpha^*,1)$, consider an optimal $(s^\prime_\alpha, S^\prime_\alpha)$
          policy for the discounted-cost criterion with the discount factor $\alpha.$
          Let $\{\alpha_n\uparrow 1\}_{n=1,2,\ldots}$ be a sequence of negative numbers with $\alpha_1> \alpha^*.$  Every sequence
          $\{(s^\prime_{\alpha_n}, S^\prime_{\alpha_n})\}_{n=1,2,\ldots}$ is
          bounded, and  each its limit point $(s^*,S^*)$ defines an
          average-cost optimal $(s^*,S^*)$ policy. Furthermore, this policy satisfies
          the optimality inequality \eqref{eqn:IC ACOE-1},
          where the function $\tilde u$ is defined in \eqref{EQN1} for an arbitrary subsequence
         $\{\alpha_{n_k}\}_{k=1,2,\ldots}$ of $\{\alpha_{n}\}_{n=1,2,\ldots}$ satisfying $(s^*,S^*)=\lim_{k\to\infty} (s^\prime_{\alpha_{n_k}}, S^\prime_{\alpha_{n_k}}).$
\end{theorem}

The following theorem states that the conditions and conclusions described in Theorem \ref{thm:acoe}  hold for the described inventory control problem.  It also states some problem-specific results.

\begin{theorem}\label{thm:inventory:acoe}
The MDP for the described inventory control problem satisfies the sufficient conditions stated
in Theorem \ref{thm:acoe}. Therefore, the conclusions of Theorem \ref{thm:acoe} hold for any sequence $\{\alpha_n\uparrow 1\}_{n=1,2,\ldots}$ of nonnegative discount factors with $\alpha_1>\alpha^*,$ that is, there exists a stationary policy $\varphi$ such that for all $x\in\X$
\begin{align}\label{eqn:IC ACOE}
  w + \tu(x)  = K I_{\{\varphi(x)>0\}} + H(x+\varphi(x)) - \bar{c}x
  =\min \{ \min_{a\geq 0}[ K + H(x+a)], H(x) \} - \bar{c}x,
\end{align}
where the function $H$ is defined in \eqref{eqn:IC ACOE-2}. In addition, the functions $\tilde u$ and $H$ are $K$-convex, continuous and inf-compact, and a stationary optimal policy $\varphi$ satisfying \eqref{eqn:IC ACOE} can be selected as an $(s^*,S^*)$ policy described in Theorem~\ref{thm:fle}.  It also can be selected  as an $(s,S)$ policy with the real numbers $S$ and $s$ satisfying \eqref{eqn:def S} and  defined in   \eqref{eqn:def s} respectively for $f(x)=H(x),$ $x\in\X.$
\end{theorem}

To prove Theorem \ref{thm:inventory:acoe}, we first state several auxiliary facts. Consider the renewal process
\begin{align*}
	\textbf{N}(t) :=  \sup \{ n=0,1,\ldots|\, \textbf{S}_n \leq t \},
\end{align*}
where $t\in\Rp,$ $\textbf{S}_0 = 0$ and $\textbf{S}_n = \sum_{j=1}^n D_j$ for $n=1,2,\ldots\ .$
Observe that since $P(D > 0) > 0,$ then $\E[\textbf{N}(t)] < +\infty,$ $t\in\Rp;$ see Resnick~\cite[Theorem 3.3.1]{Res92}.

Consider an arbitrary $\alpha\in [0,1)$ and a state $x_\alpha$ such that $u_\alpha(x_\alpha)=m_\a.$ Then, in view of \eqref{eq:boundmaxlul}, the inequalities $x^*_L\le x_\a\le x^*_U$ take place.

Define $E_y(x) := \E[h(x - \textbf{S}_{\textbf{N}(y)+1})]$ for $x\in\X,$ $y\ge 0.$  In view of Feinberg and Lewis~\cite[Lemma 6.2]{FEINBERG2015}, $E_y(x)
<+\infty.$ According to Feinberg and Lewis~\cite[inequalities (6.11), (6.17)]{FEINBERG2015}, for $x<x^\a$
\begin{align}
	u_{\a}(x) \leq K+\bar{c}(x^*_U - x),
	\label{eqn:ua ubd 1}
\end{align}
and for $x\geq x^\a$
\begin{align}
	u_{\a}(x) \leq K + (E(x) + \bar{c}\E[D])(1+\E[\textbf{N}(x-x^*_L)]),
	\label{eqn:ua ubd 2}
\end{align}
where 
$E(x) := h(x) + E_{x-x^*_L}(x).$ Let
\begin{align}
	U(x) :=
	\begin{cases}	
K+\bar{c}(x^*_U - x), & \text{if } x < x^*_L, \\
K+\bar{c}(x^*_U - x^*_L) + (E(x) + \bar{c}\E[D])(1+\E[\textbf{N}(x-x^*_L)]), & \text{if } x\ge x^*_L .
	\end{cases}
	\label{eqn:def U(x)}
\end{align}
\begin{lemma}\label{lm:U(x) bounded} The following inequalities hold for $\alpha\in [0,1):$
\begin{enumerate}[(i)]
	\item $u_\alpha(x)\le U(x) < +\infty$ for all $x\in\X;$
\item If $x_*,x\in\X$ and $x_*\le x,$ then $C(x_*,x):=\sup_{y\in [x_*,x]}  U(y)< +\infty;$
\item
 $\E[U(x-D)] < +\infty$ for all $x\in\X.$
\end{enumerate}
\end{lemma}

\begin{proof}
(i) For $x < x^*_L$  the inequality $u_\alpha(x)\le U(x)$ holds because of \eqref{eqn:ua ubd 1}.  For $x\ge x^*_L$ denote by $f$ the function added to the constant $K$ in the right-hand side of \eqref{eqn:ua ubd 2},
 \begin{equation}\label{definition of f}
 f(x):=(E(x) + \bar{c}\E[D])(1+\E[\textbf{N}(x-x^*_L)]).
 \end{equation}
For $x\ge x^*_U,$ inequality  \eqref{eqn:ua ubd 2} and the inequality $u^*_U\ge x^*_L$ imply that
 \[
 u_\a(x)\le K+f(x)\le K+\bar{c}(x^*_U - x^*_L)+f(x)=U(x),
 \]
where the first inequality is \eqref{eqn:ua ubd 2}, for $x\ge u^*_U\ge x_\a,$ and the second inequality follows from $u^*_U\ge x^*_L.$  Thus, $u_\a(x)\le U(x)$ for  $x\ge x^*_U.$

 For $x^*_L\le x<u^*_U$
\begin{align*}
	u_\alpha(x) & \le K+\max\{ \bar{c}(x^*_U-x),f(x)\} \\
	& \le K+\bar{c}(x^*_U-x)+f(x)\le  K+\bar{c}(x^*_U-x^*_L)+f(x)=U(x),
\end{align*}
where the first inequality follows from
   \eqref{eqn:ua ubd 1}, \eqref{eqn:ua ubd 2}, and $x^*_L\le x_\a\le x^*_U,$ the second inequality holds because the maximum of two nonnegative numbers is not greater than their sum, and the last inequality follows from $x^*_L\le x_\a\le x^*_U.$   In addition, $U(x)<+\infty$ because all the functions in the right-hand side of \eqref{eqn:def U(x)} take real values.

(ii) For $x< x^*_L$
\begin{equation}\label{eqcaseiifff1}
C(x_*,x)\le \sup_{y\in [x_*,x^*_L)} U(y)\le K+\bar{c}(x^*_U-x_*)<+\infty.
\end{equation}

Let $x^*_L\le x_*.$  In this case,
\[
C(x_*,x)\le C(x^*_L,x)=K+\bar{c}(x^*_U - x^*_L)+\sup_{y\in [x^*_L,x]} f(y),
\]
where the function $f$ is defined in \eqref{definition of f} and $f(y)\le (E(y) + \bar{c}\E[D])(1+\E[\textbf{N}(x-x^*_L)])$ for $y\in [x^*_L,x].$  To complete the proof of
$C(x_*,x)<+\infty$ for $x^*_L\le x_*,$ we need to show that $\sup_{y\in [x^*_L,x]} E(y)<+\infty.$  This is true because of the following reasons.  First, by Feinberg and Lewis~\cite[inequalities (6.5), (6,6), and the inequality between them]{FEINBERG2015}, for $z\ge 0$ and $y\in\X$
 \begin{equation} \label{eqEzy}
   E_z(y)\le (1+\E[ {\bf N}(z)] )\E [h(y-z-D)] + h(y).
   \end{equation}

  Therefore, for $y\in [x^*_L,x]$
\begin{align*}
  E(y)&\le (1+\E [{\bf N}(y-x^*_L)])\E[ h(x^*_L-D)]+2h(y)\\
  &\le (1+\E [{\bf N}(x-x^*_L)])\E [h(x^*_L-D)]+2\max\{h(x^*_L), h(x)\}<+\infty,
\end{align*}
where the first inequality follows from the definition of the function $E(\cdot),$ introduced after \eqref{eqn:ua ubd 2}, and from \eqref{eqEzy}.  The second inequality follows from the convexity of $h$ and from $x^*_L\le y\le x.$ Thus, for $x^*_L\le x_*$
\begin{equation}\label{eqcaseiifff2}
C(x_*,x)\le C(x^*_L,x)=K+\bar{c}(x^*_U - x^*_L)+\sup_{y\in [x^*_L,x]} f(y)<+\infty.
\end{equation}

Now consider arbitrary $x_*,x\in\X$ such that $x_*\le x.$  Choose $z_*,z\in\X$ such that $z_*<\min\{x_*, x^*_L\} $ and $z>\max\{x,x^*_L\}.$  Then
\[ C(x_*,x)\le C(z_*,z)\le  \max\{ \sup_{y\in [z_*,x^*_L)} \{ U(y)\} ,C(x^*_L,z)\}<+\infty,\]
where the first inequality follows from $[x_*,x]\subset [z_*,z],$ the second inequality follows from $[z_*,z]=[z_*,x^*_L)\cup [x^*_L,z],$ and the last one follows from
\eqref{eqcaseiifff1} and from \eqref{eqcaseiifff2}.

(iii) Let us define $C(x_*,x)=0$ for $x_*,x\in \X$ and $x_*>x.$ For $x\in\X$
  \begin{align*}
  & \E[U(x-D)]=\E[U(x-D)I_{\{x-D<x^*_L\} }]+\E[U(x-D)I_{\{x^*_L\le x-D\le x\} }] \\
  \le &\E[(K+\bar{c}(x^*_U-x+D))I_{\{x-D<x^*_L\} }]+\E[C(x^*_L,x)I_{\{x^*_L\le x-D\le x\} }] \\
 \le & (K+\bar{c}(x^*_U-x))P(D>x-x^*_L)+ \bar{c}\E[D] +C(x^*_L,x)<+\infty,
  \end{align*}
where the first equality holds because $D$ is a nonnegative random variable, the first inequality follows from  the definitions of the functions $U$ and $C,$ the second inequality holds because an expectation of an indicator of an event is its probability and because the random variable $D$ and the constant $C(x^*_L,x)$ are nonnegative, and the last inequality follows from $\E[D]<\infty$ and from Lemma~\ref{lm:U(x) bounded}(ii).
\end{proof}

The next lemma establishes the equicontinity on $\X$ of a family of discounted relative value functions $\{u_{\a_n}\}_{n=1,2,\ldots}$ with $\alpha_n\uparrow 1.$

\begin{lemma}\label{lm:u a equicont}
	For each sequence $\{\alpha_n\uparrow 1\}_{n=1,2,\ldots}$ of nonnegative discount factors with $\alpha_1>\alpha^*,$ the family $\{u_{\a_n}\}_{n=1,2,\ldots}$ is equicontinuous on $\X.$
\end{lemma}

\begin{proof}  Before providing the proof, we would like to describe its main idea.   It is based on  estimating the difference between the total discounted costs incurred when the process starts from two states, $z_1$ and   $z_2,$ when the distance between $z_1$ and $z_2$ is small.  Let $z_1<z_2.$  This estimation is trivial when $z_2\le s_{\alpha_n}$ because the function $u_{\a_n}(x)$ is linear on $(-\infty, s_{\alpha_n}].$  By using Lemma~\eqref{lm:U(x) bounded}(ii), it is possible to derive such estimation for $z_1\le  s_{\alpha_n}< z_2.$ For $z_1>  s_{\alpha_n},$   the estimation consists of two parts: (i) the difference between the total holding costs incurred until the process, that starts at $z_1,$ reaches the set $(-\infty,s_{\a_n}],$ and this difference is small because of the Lipshitz continuity of the convex function $\E[h(x-D]$ on a bounded interval and because the average number of jumps is finite; (ii) the difference between the total costs incurred after the process, that starts at $z_1,$ reaches $(-\infty,s_{\a_n}],$ and this difference is small because it is bounded by the differences of the total costs for the two cased $z_2\le s_{\alpha_n}$ and $z_1\le  s_{\alpha_n}< z_2$ described above. Now we start the proof.

	The discounted-cost optimality equations \eqref{eqn:va} and the optimality of $(s_{\a_n},S_{\a_n})$ policies, stated in Theorem~\ref{thm:fli}, imply that the function $v_{\a_n} (x )$ is linear, when $x\le s_{\alpha_n},$ and
\begin{align}
	v_{\a_n} (x ) =
	\begin{cases}	
		\bar{c}(s_{\a_n} - x) + v_{\a_n}(s_{\a_n}), & \text{if } x \leq s_{\a_n}, \\
		\tilde{h}(x) + \a_n\E[v_{\a_n}(x-D)], & \text{if } x \geq s_{\a_n} ,
	\end{cases}
	\label{eqn:va two cases}
\end{align}
where $\tilde{h}(x) := \E[h(x-D)] < +\infty$ is convex in $x$ on $\X.$
According to Theorem \ref{thm:fle}, since each sequence $\{(s_{\a_n}, S_{\a_n})\}_{n=1,2,\ldots}$ is bounded, then
there exist a constant $b>0$ such that $s_{\a_n}\in (-b,b)$ and $S_{\a_n}\in (-b,b),$ $n=1,2,\ldots,$  and a constant $\delta_0 > 0$ such that $-b\leq s_{\a_n} - \delta_0 < s_{\a_n} + \delta_0\leq b,$ $n=1,2,\ldots\ .$

Consider $z_1,z_2\geq s_{\a_n}.$ Without loss of generality, assume that $z_1<z_2.$ According to \eqref{eqn:va two cases}, $v_{\a_n}(x) = \E[\sum_{j=1}^{\textbf{N}(x-s_{\a_n})+1} \a_n^{j-1}\tilde{h}(x-\textbf{S}_{j-1}) + \a_n^{\textbf{N}(x-s_{\a_n})+1}  v_{\a_n} (x-\textbf{S}_{\textbf{N}(x-s_{\a_n})+1})]$ for $x\geq s_{\a_n}.$ Therefore, for $n=1,2,\ldots$
\begin{align}
\begin{split}
	& |u_{\a_n}(z_1) - u_{\a_n}(z_2)| = |v_{\a_n}(z_1) - v_{\a_n}(z_2)|  \\
	= & {\Big |}\E[\sum_{j=1}^{\textbf{N}(z_1-s_{\a_n})+1} \a_n^{j-1}(\tilde{h}(z_1-\textbf{S}_{j-1}) - \tilde{h}(z_2-\textbf{S}_{j-1}))  \\
	+ & \a_n^{\textbf{N}(z_1-s_{\a_n})+1} (v_{\a_n}(z_1-\textbf{S}_{\textbf{N}(z_1-s_{\a_n})+1}) - v_{\a_n}(z_2-\textbf{S}_{\textbf{N}(z_1-s_{\a_n})+1}) )] {\Big |}  \\
	\leq & \E[\sum_{j=1}^{\textbf{N}(z_1-s_{\a_n})+1} |\tilde{h}(z_1-\textbf{S}_{j-1}) - \tilde{h}(z_2-\textbf{S}_{j-1})|]   \\
	+ &  \E[ |u_{\a_n}(z_1-\textbf{S}_{\textbf{N}(z_1-s_{\a_n})+1}) - u_{\a_n}(z_2-\textbf{S}_{\textbf{N}(z_1-s_{\a_n})+1})|] ,
\end{split}
	\label{eqn:ubd2}
\end{align}
where the inequality holds because of $\a_n < 1,$  the change of the expectations and the absolute values, and because the sum of absolute values is greater or equal than the absolute value of the sum.

Consider $\epsilon > 0.$  Define  a positive number $\bar{N} := \E[\textbf{N}(z_1+b)]+1 < +\infty.$ Since $b>-s_\alpha,$ then
$\E[\textbf{N}(z_1-s_{\a_n})]+1 \leq \bar{N}.$
Since the function $\tilde{h}(x)$ is convex on $\R$, then it is Lipschitz continuous on $[-b,z_2];$ see
Hiriart-Urruty and Lemar\'{e}chal~\cite[Theorem 3.1.1]{hl93}. Since Lipschitz continuity implies uniformly continuity, then there exists $\delta_1 \in (0,\delta_0)$ such that for $x,y\in [-b,z_2]$ satisfying $|x-y| < \delta_1,$ $|\tilde{h}(x) - \tilde{h}(y)| < \dfrac{\epsilon}{2\bar{N}}.$ Therefore, for $s_{\a_n}\leq z_1<z_2$ satisfying $|z_1 - z_2|<\delta_1$
\begin{align}
	|\tilde{h}(z_1-\textbf{S}_j) - \tilde{h}(z_2-\textbf{S}_j)| < \frac{\epsilon}{2\bar{N}}, \qquad j = 0,1,\ldots, \textbf{N}(z_1-s_{\a_n}) ,
	\label{eqn:lip cont}
\end{align}
and
\begin{align}
\begin{split}
	& \E[\sum_{j=1}^{\textbf{N}(z_1-s_{\a_n})+1} |\tilde{h}(z_1-\textbf{S}_{j-1}) - \tilde{h}(z_2-\textbf{S}_{j-1})|] 	
	\leq \E[\sum_{j=1}^{\textbf{N}(z_1-s_{\a_n})+1} \frac{\epsilon}{2\bar{N}} ]   \\
	 = & ( \E[\textbf{N}(z_1-s_{\a_n})] + 1) \frac{\epsilon}{2\bar{N}}
	\leq  \frac{\epsilon}{2}.
	\end{split}
	\label{eqn:ubd2-1}
\end{align}
where the first inequality follows from \eqref{eqn:lip cont} 
and the last inequality holds because of $\E[\textbf{N}(z_1-s_{\a_n})]+1 \leq \bar{N}.$

Additional arguments are needed to estimate the last term in \eqref{eqn:ubd2}.
Next we prove that there exists $\delta_2\in (0,\delta_1)$ such that for $x\in [s_{\a_n},s_{\a_n}+\delta_2],$
\begin{align}
	|u_{\a_n}(x) - u_{\a_n}(s_{\a_n})| < \frac{\epsilon}{4},  \qquad n=1,2,\ldots\ .
	\label{eqn:ubd ux uan}
\end{align}

Let $x\ge s_\a.$ Then formula \eqref{eqn:va two cases} implies 
\begin{equation}v_{\a_n} (x ) =
		\tilde{h}(x) + \a_n\E[v_{\a_n}(x-D)] \label{firstfrom423}
\end{equation}
and
\begin{align}
\begin{split}
&\E[v_{\a_n}(x-D)] =
\PR(D\geq x - s_{\a_n}) \E[\bar{c}(s_{\a_n} - x + D)|D\geq x - s_{\a_n}]  \\
&+ \PR(0<D< x - s_{\a_n}) \E[v_{\a_n}(x-D)|0<D< x - s_{\a_n}]
+ \PR(D = 0) v_{\a_n}(x)
\end{split}
\label{secondfrom423}
\end{align}
Formulas \eqref{firstfrom423} and \eqref{secondfrom423} imply
\begin{align}
[1&-\a_n\PR(D=0)]	v_{\a_n}(x)  =  \tilde{h}(x) + \a_n (\PR(D\geq x - s_{\a_n}) \E[\bar{c}(s_{\a_n} - x + D)|D\geq x - s_{\a_n}] \nonumber \\
& + \PR(0<D< x - s_{\a_n}) \E[v_{\a_n}(x-D)|0<D< x - s_{\a_n}]).
	\label{eqn: ua-1}
\end{align}
 Therefore, since $u_{\a_n}(y_1) - u_{\a_n}(y_2) = v_{\a_n}(y_1) - v_{\a_n}(y_2)   $ for all $y_1,y_2\in\X,$ for $x\in [s_{\a_n},s_{\a_n}+\delta_1]$ and for $n=1,2,\ldots$
\begin{align}
\begin{split}
	& [1-\a_n\PR(D=0)]|u_{\a_n}(x) - u_{\a_n}(s_{\a_n})| = [1-\a_n\PR(D=0)]|v_{\a_n}(x) - v_{\a_n}(s_{\a_n})|  \\
	&=   \Big |\tilde{h}(x)-\tilde{h}(s_{\a_n}) +\a_n\PR(D\geq x - s_{\a_n})\bar{c}(s_{\a_n} - x)  \\ & +\a_n\PR(0<D< x - s_{\a_n})\E[u_{\a_n}(x-D)
	-u_{\a_n}(s_{\a_n}-D)|0<D< x - s_{\a_n}]   \Big |  \\
	&\leq  |\tilde{h}(x)-\tilde{h}(s_{\a_n})|+\bar{c}(x-s_{\a_n})+ 2 \PR(0<D< x-s_{\a_n}) C(-b,b),
\end{split}
\label{eqn:dif ua ubd}
\end{align}
where the nonnegative function $C$ is defined in Lemma~\ref{lm:U(x) bounded}. Let us define $L:= (1-\PR(D=0))^{-1},$ and $Q(x,s_{\alpha_n}):=\PR(0<D< x-s_{\a_n}).$   Recall that $\PR(D>0)>0, $ which is equivalent to $\PR(D=0)<1.$ Since $(1-\a_n\PR(D=0))^{-1}\le L,$  formula \eqref{eqn:dif ua ubd} implies that for  $n=1,2,\ldots$
\begin{equation}
|u_{\a_n}(x) - u_{\a_n}(s_{\a_n})|\le L(|\tilde{h}(x)-\tilde{h}(s_{\a_n})|+\bar{c}(x-s_{\a_n}) + 2 Q(x,s_{\alpha_n}) C(-b,b)).
\end{equation}
Since the function $\tilde h$ is convex, it is Lipshitz continuous on $[-b,b].$ Therefore, all three summands in the right-hand side of the last equations converge uniformly in $n$ to 0 as $x\downarrow s_{\a_n}.$  Therefore, there exists $\delta_2\in (0,\delta_1)$ such that \eqref{eqn:ubd ux uan} holds for all $x\in [s_{\a_n},s_{\a_n}+\delta_2].$

Since $u_{\a_n} (x) = \bar{c}(s_{\a_n} - x)  + u_{\a_n} (s_{\a_n})$ for all $x\leq s_{\a_n},$ then for all $x,y\le s_{\a_n}$ 
\begin{align}
	|u_{\a_n}(x) - u_{\a_n}(y)| =\bar{c}|x-y|< \frac{\epsilon}{4},  \qquad n=1,2,\ldots\ ,
	\label{eqn:ubd u x < san}
\end{align}
for $|x-y|<\frac{\epsilon}{4{\bar c}}.$ Let $\delta_3:=\min\{\frac{\epsilon}{4{\bar c}},\delta_2\}.$ Then \eqref{eqn:ubd u x < san} holds for $|x-y|<\delta_3.$

For $x \leq s_{\a_n} \leq y$ satisfying $|x-y|< \delta_3$
\begin{align}
	|u_{\a_n}(x) - u_{\a_n}(y)| \leq |u_{\a_n}(x) - u_{\a_n}(s_{\a_n})| + |u_{\a_n}(s_{\a_n}) - u_{\a_n}(y)|
	< \frac{\epsilon}{2},
	\label{eqn:ubd x < san < y}
\end{align}
where the first inequality is the triangle property and the second one follows from \eqref{eqn:ubd ux uan} and \eqref{eqn:ubd u x < san}.
Therefore, \eqref{eqn:ubd ux uan}, \eqref{eqn:ubd u x < san} and \eqref{eqn:ubd x < san < y} imply that $|u_{\a_n}(x) - u_{\a_n}(y)| <\frac{\epsilon}{2}$ for all $x,y \leq s_{\a_n} + \delta_3$ satisfying $|x-y|< \delta_3.$
Then for $|z_1-z_2|< \delta_3$ with probability 1
\[
|u_{\a_n}(z_1-\textbf{S}_{\textbf{N}(z_1-s_{\a_n}) + 1}) - u_{\a_n}(z_2-\textbf{S}_{\textbf{N}(z_1-s_{\a_n}) + 1})|  < \frac{\epsilon}{2},  \qquad n=1,2,\ldots\ ,
\]
and therefore
\begin{align}
	 \E[|u_{\a_n}(z_1-\textbf{S}_{\textbf{N}(z_1-s_{\a_n}) + 1}) - u_{\a_n}(z_2-\textbf{S}_{\textbf{N}(z_1-s_{\a_n}) + 1})|]  < \frac{\epsilon}{2},  \qquad n=1,2,\ldots\ .
	\label{eqn:ubd2-2}
\end{align}
Formulae \eqref{eqn:ubd2}, \eqref{eqn:ubd2-1}. and  \eqref{eqn:ubd2-2} imply that for $z_1,$ $z_2\geq s_{\a_n}$ satisfying $|z_1-z_2|< \delta_3$
\begin{align}
	|u_{\a_n}(z_1) - u_{\a_n}(z_2)] | < \epsilon, \qquad n = 1,2,\ldots\ .
	\label{eqn:ubd x12>san}
\end{align}

 Therefore, \eqref{eqn:ubd u x < san}, \eqref{eqn:ubd x < san < y}, and \eqref{eqn:ubd x12>san}  imply that for each $x\in\X$
\begin{align*}
	|u_{\a_n}(x) - u_{\a_n}(y)] | < \epsilon, \qquad n = 1,2,\ldots,
\end{align*}
if $|x-y|<\delta_3, $ which means that the family $\{u_{\a_n}\}_{n=1,2,\ldots}$ is equicontinuous on $\X.$
\end{proof}

\begin{proof}[Proof of Theorem \ref{thm:inventory:acoe}]
Since $\int_{\X} U(y)q(dy|x,a)  = \E[U(x+a-D)],$ where the function $U$ is defined in \eqref{lm:U(x) bounded}, then, in view of Lemma \ref{lm:U(x) bounded}(iii), $\int_{\X} U(y)q(dy|x,a) < +\infty$ for all $x\in\X$ and $a\in \A.$ According to Lemma~\ref{lm:u a equicont}, the family $\{u_{\a_n}\}_{n=1,2,\ldots}$ is equicontinuous on $\X.$ Therefore, Theorem \ref{thm:acoe} implies that there exists a subsequence $\{ \a_{n_k} \}_{k=1,2,\ldots}$ of $\{ \a_n \}_{n=1,2,\ldots}$ such that there exists a policy $\varphi$ satisfying  ACOE \eqref{EQN11} with $\tu$ defined in \eqref{EQN1} for the sequence $\{\a_{n_k}\}_{k=1,2,\ldots},$ the function $u_{\a_{n_k}}$ converges pointwise to $\tu,$ and the function $\tu$ is continuous.

According to Feinberg and Lewis~\cite[Theorem 6.10]{FEINBERG2015}, the $(s^*,S^*)$ policy satisfies the ACOI with $\tu$ defined in \eqref{EQN1} for the sequence $\{\a_{n_k}\}_{k=1,2,\ldots}.$ Since the ACOE holds with $\tu$ defined in \eqref{EQN1} for the sequence $\{\a_{n_k}\}_{k=1,2,\ldots},$ then the $(s^*,S^*)$ policy satisfies the ACOE.

Next we show that the functions $\tu$ and $H$ are $K$-convex and inf-compact.  Since the cost function $c$ is inf-compact, the function $\tilde u$ is inf-compact; see Feinberg et al.~\cite[Theorem 3 and Corollary 2]{FEINBERG2012}.
According to Feinberg and Lewis~\cite[Lemma 6.8]{FEINBERG2015}, the functions $v_{\a_n}$ are $K$-convex.  Therefore the functions $u_{\a_n}$ are $K$-convex. Since $u_{\a_{n_k}}$ converges pointwise to $\tu,$ then the function $\tu$ is $K$-convex. The function $H$ is $K$-convex because, in view of \eqref{eqn:IC ACOE-2}, it is a sum of a linear, convex, and $K$-convex functions.

Since the $(s^*,S^*)$ policy satisfies the ACOE \eqref{eqn:IC ACOE} with $\tu$ defined in \eqref{EQN1} for the sequence $\{\a_{n_k}\}_{k=1,2,\ldots},$ then $\tu(x) = K + H(S^*) - \bar{c}x - w,$ for all $x < s.$ Therefore, for $x < s,$
\begin{align}
\begin{split}
	H(x) & = \bar{c}x +\E[h(x-D)] + \E[\tu (x-D)]  \\
	& = \bar{c}x +\E[h(x-D)] + K + H(S^*) - \bar{c}x + \bar{c}\E[D] - w  \\
	& = \E[h(x-D)] +  K + H(S^*) + \bar{c}\E[D] - w .
\end{split}
\label{eqn:H(x) x < s}
\end{align}
Since $\E[h(x-D)]\to +\infty$ as $x\to -\infty,$ then \eqref{eqn:H(x) x < s} implies that $H(x)$ tends to $+\infty$ as $x\to -\infty.$
Since $h$ and $\tu$ are nonnegative, then \eqref{eqn:IC ACOE-2} implies that $H(x) \geq \bar{c}x \to +\infty$ as $x\to +\infty.$ Therefore, $H(x)\to +\infty$ as $|x|\to +\infty.$

 Since $\tu$ is continuous and $(y-D)$ converges weakly to $(x-D)$ as $y\to x,$ then $\E[\tu(x-D)]$ is lower semi-continuous. Since $\E[h(x-D)]$ is convex on $\X$ and hence continuous, $\bar{c}x$ is continuous and $\E[\tu(x-D)]$ is lower semi-continuous, then $H$ is lower semi-continuous. Therefore, since $H(x)$ tends to  $+\infty$ as $|x|\to +\infty,$ then $H$ is inf-compact.

According to the statements following Feinberg and Lewis~\cite[Lemma 6.7]{FEINBERG2015}, since $H$ is $K$-convex, inf-compact, and tends to $+\infty$ as $|x|\to +\infty,$
then an $(s,S)$ policy, with the real numbers $S$ and $s$ satisfying \eqref{eqn:def S} and  defined in   \eqref{eqn:def s} respectively for $f(x)=H(x),$ $x\in\X,$ is optimal.

Now we prove that the function $H$ is continuous. Let us fix an arbitrary $y\in\X.$ Define the following function
\begin{align*}
	\bar{H}(x) =
	\begin{cases}
		\tu(x) + \bar{c}x,  		& \text{if } x\leq y + 1 , \\
		\tu(y+1) + \bar{c}(y+1),  	& \text{if } x > y + 1 . \\
	\end{cases}
\end{align*}
Since the functions $\tu (x)$ and $\bar{c}x$ are continuous, then the function $\bar{H} (x)$ is continuous.
In view of \eqref{eqn:IC ACOE}, the function $\bar{H} (x)$ is bounded on $\X.$ Therefore,
\begin{align}
	\lim_{z\to y} \{ \E[h(z-D)] + \E[\bar{H} (z-D)] \} = \E[h(y-D)] + \E[\bar{H} (y-D)] ,
	\label{eqn:limit bar H}
\end{align}
where the equality holds since the function $\E[h(x-D)]$ is convex
on $\X$ and hence it is continuous, and $z-D$ converges weakly to $y-D$ as $z\to y$ and the function
$\bar{H} (x)$ is continuous and bounded.

Observe that $H (x) = \E[h(x-D)] + \E[\bar{H} (x-D)] + \bar{c}\E[D]$ for all $x\leq y + 1.$
Therefore, \eqref{eqn:limit bar H} implies that $\lim_{z\to y}H (z) = H (y).$ Thus the function
$H (x)$ is continuous.
\end{proof}

\begin{corollary}\label{thm:ordering at sa}

For the $(s,S)$ policy defined in Theorems \ref{thm:inventory:acoe}, consider the stationary policy $\varphi$ coinciding with this policy at all $x\in\X,$ except $x=s,$ and with   $\varphi(s)=S-s.$  Then the stationary policy $\varphi$ also satisfies the optimality equation~\eqref{eqn:IC ACOE}, and therefore the policy $\varphi$ is average-cost optimal.

\end{corollary}

\begin{proof} Since the proof of the optimality of $(s,S)$ policies is based on the fact that
$K+H(S)<H(x),$ if $x<s,$ and $K+H(S)\ge H(x),$ if $x\ge s.$  Since the function $H$ is continuous, we have that $K+H(S)=H(s).$ Thus both actions are optimal at the state $s.$ \end{proof}

{\bf Acknowledgement.}
 This research  was partially supported by NSF grants CMMI-1335296 and CMMI-1636193.

\end{document}